\documentclass[11pt,leqno]{amsart}
\usepackage{amsthm,amsfonts,amssymb,amsmath,oldgerm}
\numberwithin{equation}{section}
\usepackage{fullpage}
\renewcommand\d{\partial}

\renewcommand\d{\partial}

\newcommand\R{\mathbb R}

\newcommand\br{\begin{remark}}
\newcommand\er{\end{remark}}
\newcommand\bp{\begin{pmatrix}}
\newcommand\ep{\end{pmatrix}}
\newcommand{\be}{\begin{equation}}
\newcommand{\ee}{\end{equation}}
\newcommand\ba{\begin{equation}\begin{aligned}}
\newcommand\ea{\end{aligned}\end{equation}}

\newcommand{\bap}{\begin{app}}
\newcommand{\eap}{\end{app}}
\newcommand{\begs}{\begin{exams}}
\newcommand{\eegs}{\end{exams}}
\newcommand{\beg}{\begin{example}}
\newcommand{\eeg}{\end{exaplem}}
\newcommand{\bpr}{\begin{proposition}}
\newcommand{\epr}{\end{proposition}}
\newcommand{\bt}{\begin{theorem}}
\newcommand{\et}{\end{theorem}}
\newcommand{\bc}{\begin{corollary}}
\newcommand{\ec}{\end{corollary}}
\newcommand{\bl}{\begin{lemma}}
\newcommand{\el}{\end{lemma}}
\newcommand{\bd}{\begin{definition}}
\newcommand{\ed}{\end{definition}}
\newcommand{\brs}{\begin{remarks}}
\newcommand{\ers}{\end{remarks}}

\newcommand{\CalF}{\mathcal{F}}

\newcommand{\N}{\mathcal{N}}

\newcommand{\Id}{{\rm Id }}
\newcommand{\Range}{{\rm Range }}
\newcommand{\suppt}{{\rm suppt }}

\newtheorem{theorem}{Theorem}[section]
\newtheorem{proposition}[theorem]{Proposition}
\newtheorem{corollary}[theorem]{Corollary}
\newtheorem{lemma}[theorem]{Lemma}

\theoremstyle{remark}
\newtheorem{remark}[theorem]{Remark}
\theoremstyle{definition}
\newtheorem{definition}[theorem]{Definition}

\newtheorem{example}[theorem]{Example}

\newcommand\cT{{\mathcal T}}
\newcommand\cS{{\mathcal S}}
\newcommand\cR{{\mathcal R}}
\newcommand\cQ{{\mathcal Q}}

\newcommand{\RM}{\mathbb{R}}
\newcommand{\ZM}{\mathbb{Z}}

\newcommand{\CM}{\mathbb{C}}
\newcommand{\NM}{\mathbb{N}}

\newcommand{\beq}{\begin{equation}}
\newcommand{\eeq}{\end{equation}}
\newcommand{\ks}{k_*}

\title{
Nonlocalized modulation of periodic reaction diffusion waves:
The Whitham equation
}

\author{ Mathew A. Johnson}
\address{University of Kansas, Lawrence, KS 66045}
\email{matjohn@math.ku.edu}
\thanks{ Research of M.J. was partially supported by an NSF Postdoctoral Fellowship under NSF grant DMS-0902192}
\author{Pascal Noble}
\address{Universit\'e Lyon I, Villeurbanne, France}
\email{noble@math.univ-lyon1.fr}
\thanks{Research of P.N. was partially supported by the French ANR Project no.
ANR-09-JCJC-0103-01}
\author{L.Miguel Rodrigues}
\address{Universit\'e Lyon 1, Villeurbanne, France}
\email{ rodrigues@math.univ-lyon1.fr}
\thanks{ Stay of M.R. in Bloomington was supported by %Pascal Noble's French ANR grant SWECF.}
French ANR project no. ANR-09-JCJC-0103-01}
\author{Kevin Zumbrun}
\address{Indiana University, Bloomington, IN 47405}
\email{kzumbrun@indiana.edu}
\thanks{Research of K.Z. was partially supported
under NSF grant no. DMS-0300487}
\begin{document}

\begin{abstract}
In a companion paper, we established
nonlinear stability with detailed diffusive rates of decay
of spectrally stable periodic traveling-wave solutions of reaction
diffusion systems under small perturbations consisting of a nonlocalized
modulation plus a localized ($L^1$)
perturbation.
Here, we determine time-asymptotic behavior under such perturbations,
showing
that solutions consist to leading order of a modulation
whose parameter evolution is governed by an associated Whitham averaged equation.
\end{abstract}
\date{\today}
\maketitle

\section{Introduction}\label{s:introduction}
In this, the second part of a two-part series of papers,
we determine the time-asymptotic behavior
of spectrally stable periodic traveling-wave solutions of reaction
diffusion systems under small perturbations consisting of a nonlocalized
modulation plus a localized
($L^1$)
perturbation, showing that solutions
consist of an (in general nonlocalized) modulation governed by the formal second-order Whitham
averaged equations plus a faster-decaying localized residual.

In the first part, \cite{JNRZ1}, we established nonlinear stability under such perturbations,
together with detailed rates of decay, using a refined version
of the argument used in \cite{JZ1} to show stability under localized
perturbations. In each of these cases, the basic approach was to
introduce a phase \emph{via} a phase-dependent change of variables, then
separate out from the linearized solution operator slowly-decaying terms
corresponding to linear phase shifts, and
estimate separately the modulational and nonmodulational parts of the solution.
While we reduced at first order the long-time dynamics of the solution to
that of phase shifts, no information about the behavior of the phase was given besides decay rates.

Here, we show that, using the same basic linear estimates as in \cite{JNRZ1}, but a one-order-higher decomposition of the solution operator, we may
describe at a higher precision the asymptotic behavior in terms of a full modulation instead of just a phase shift.
This enables us, by a further linear estimate in the spirit of \cite{JNRZ1},
to prove that the principal part of the local wave number appearing in this
modulation obeys a convected Burgers equation that is
asymptotically equivalent
to a solution of an associated
{\it Whitham modulation equation} of a type derived formally
in, e.g., \cite{S1,DSSS,W}.

This not only gives rigorous validation of the formal Whitham approximation
in the strong sense of showing that it describes time-asymptotic behavior,
but, through the explicit prescription of initial data coming from our analysis,
also gives new information of predictive value not available from the formal
asymptotic derivation. Indeed, as discussed in \cite{BJNRZ}, the connection between
initial perturbation and initial values for the Whitham equations has remained for a long time somewhat mysterious. As a side-consequence, we show that the decay rates of \cite{JNRZ1} are sharp.

\medskip
We recall first the nonlinear stability result of \cite{JNRZ1}.
Consider a periodic traveling-wave solution
$u(x,t)=\bar u(\ks(x-ct))$
of reaction diffusion
system
$u_t=u_{xx} + f(u)$, or, equivalently, a standing-wave solution
$u(x,t)=\bar u(x)$ of
\be\label{rd}
\ks {u}_t=\ks^2{u}_{xx}+f({u})
-\omega_0
{u}_x,
\ee
where $\omega_0(k_*):=-k_* c(k_*)$ is the temporal frequency,
$c$ is the speed of the original traveling wave,
and wave number $k_*$ is chosen so that
\be\label{per}
\bar u(x+1)=\bar u(x).
\ee
Here and throughout the paper, all periodic functions are assumed to be
periodic of period one.

We make the following standard genericity assumptions:
\begin{enumerate}
  \item[(H1)] $f\in C^K(\RM^n)$,
($K\ge3$).
  \item[(H2)]  Up to translation, the set of $1$-periodic solutions of \eqref{rd} (with $k$ replacing $\ks$) in the vicinity of $\bar{u}$, $k=\ks$,
forms a smooth $1$-dimensional manifold
$\{\bar{u}(\cdot;k)\}=\{\bar{u}^k(\cdot)\}$, $c=c(k)$.
\end{enumerate}

Linearizing \eqref{rd} about $\bar u$
yields the periodic coefficient
equation
\be\label{lin}
\ks v_t=\ks Lv:= (\ks^2\partial_x^2
-\omega_0(\ks)
\partial_x +b)v,
\qquad
b(x):=df(\bar u(x)).
\ee
Introducing the one-parameter family of closed Floquet operators
\be\label{Lxi}
\ks L_\xi:= e^{-i\xi x}\ks L e^{i\xi x}=
\ks^2(\partial_x +i\xi)^2
-\omega_0(\ks)
(\partial_x+i\xi) + b
\ee
acting on $L^2_{\rm per}([0,1])$ with densely defined domains $H^2_{\rm per}([0,1])$,
determined by the defining relation
\be\label{defrel}
L (e^{i\xi x}g)= e^{i\xi x} (L_\xi g),
\ee
we define following \cite{S1,S2}
the standard {\it diffusive spectral stability} conditions:
\begin{enumerate}
  \item[(D1)]
$\sigma(L)\subset \{\lambda\ |\ \Re \lambda<0 \}\cup\{0\}$.
  \item[(D2)]
There exists a constant $\theta>0$ such that $\sigma(L_{\xi})\subset\{\lambda\ |\ \Re \lambda\leq-\theta|\xi|^2\}$
for each $\xi\in[-\pi,\pi)$.
  \item[(D3)]
$\lambda=0$ is a simple eigenvalue of $L_0$.\footnote{
$L_0$ has always at least the translational
zero-eigenfunction $\bar u'$.}
\end{enumerate}
Notice, in (D1) above, we consider $L$ as a closed operator on $L^2(\RM)$ with densely defined domain $H^2(\RM)$.

The following stability result was established in \cite{JNRZ1},
generalizing results of \cite{S1,S2,JZ1}.
Here, and throughout the paper,
given two real valued functions $A$ and $B$, we say
that $A\lesssim B$,
or that for every $x\in\textrm{dom}(A)\cap\textrm{dom}(B)$, $A(x)\lesssim B(x)$,
if there exists a constant $C>0$ such that $A(x)\leq CB(x)$ for each $x\in\textrm{dom}(A)\cap\textrm{dom}(B)$.
Even in a chain of inequalities, we will also feel free to denote by $C$ harmless constants with different values.

\begin{proposition}[\cite{JNRZ1}]\label{oldmain}
Let $K\ge 3$.
Assuming (H1)-(H2) and (D1)-(D3),
let
$$
E_0:=\big\|\tilde u_0(\cdot-h_0(\cdot))-\bar u(\cdot)\big\|_{L^1(\RM)\cap H^K(\RM)}
+\big\|\partial_x h_0\big\|_{L^1(\RM)\cap H^K(\RM)}
$$
be sufficiently small, for some choice of
phase modulation $h_0$
such that $h_0(-\infty)=-h_0(\infty)$.\footnote{
This normalization may be achieved without loss of generality
by a shift in $\bar u$; see Remark \ref{phase_wavenumber_rmk}.}
Then, there exists a global solution $\tilde u(x,t)$ of \eqref{rd}
with initial data $\tilde u_0$ and a phase function
$\psi(x,t)$ such that, for $t>0$ and $2\le p \le \infty$,
\ba\label{mainest}
\big\|\tilde u(\cdot-\psi(\cdot,t), t)-\bar u(\cdot)\big\|_{L^p(\RM)},
\quad \big\|\nabla_{x,t} \psi(\cdot,t) \big\|_{W^{K+1,p}(\RM)}
&\lesssim E_0 (1+t)^{-\frac{1}{2}(1-1/p)},\\
\big\|\tilde u(\cdot-\psi(\cdot,t), t)-\bar u(\cdot)\big\|_{H^K(\RM)}
&\lesssim E_0 (1+t)^{-\frac{1}{4}},\\
\ea
and
\ba\label{andpsi}
\big\|\tilde u(\cdot , t)-\bar u(\cdot)\big\|_{L^\infty(\RM)}, \quad
\big\| \psi(\cdot,t) \big\|_{L^\infty(\RM)} &\lesssim E_0.
\ea
In particular, $\bar u$ is nonlinearly
(boundedly) stable in $L^\infty(\RM)$
with respect to initial perturbations
$v_0=\tilde u_0-\bar u$
for which $\|v_0\|_{E}:=\inf_{\partial_x h_0\in L^1(\RM) \cap H^K(\RM)} E_0$ is sufficiently
small.
\end{proposition}

\medskip
Recall now the formal, Whitham equation, as derived in various contexts
and to varying degrees of accuracy in \cite{W,HK,Se,DSSS,NR1,NR2}.
By translation-invariance of the underling equations,
\be\label{trans}
L_0\bar u'=0,
\ee
so that by (D3) the zero-eigenspace of $L_0$ is exactly $\Range\{\bar u'\}$.
Denote by $\bar u^{ad}$ the left, or adjoint, zero eigenfunction
of $L_0$.
Fixing $\bar u$, $k_*$, introduce the parametrization
\be\label{param}
\bar u^k(x-\beta)=\bar u(k,x-\beta)
\ee
of nearby periodic traveling waves, i.e., $1$-periodic solutions of
$
\omega_0(k)\bar u^k_x-k^2\bar u^k_{xx}-f(\bar u^k)=0,
$
$\omega_0(k)=-kc(k)$, for definiteness chosen in such a way that
\be\label{gauge}
\langle \bar u^{ad}(k),\d_k\bar u(k)\rangle_{L^2([0,1])}=0.
\ee
(This may be achieved by an appropriate translation, by the fact that
$\langle \bar u^{ad},\bar u'\rangle_{L^2([0,1])} \ne 0$.)
Then, the formal approximate solution of $u_t-u_{xx}-f(u)=0$ obtained by a nonlinear WKB expansion is
\be\label{wref}
u(x,t)\approx
\bar u^{\kappa(x,t)}(\Psi(x,t)),
\ee
where the wave number
$\kappa:=\Psi_x$
satisfies the Whitham equation
(viscous scalar conservation law)
\be\label{whitham}
\kappa_t-(\omega_0(\kappa))_x=(d(\kappa)\kappa_x)_x
\ee
or equivalently, the phase $\Psi$
satisfies its integral (viscous Hamilton--Jacobi equation)
\be\label{hj}
\Psi_t-\omega_0(\Psi_x)=d(\Psi_x)\Psi_{xx}
\ee
where, taking \eqref{gauge} into account,
$
d(k)\ =\ 1+2k\langle u^{ad}(k),\d_k\bar u'(k)\rangle_{L^2([0,1])} ,
$
and $\omega_0(k)=-c(k)k$ is the nonlinear dispersion relation
determined by the manifold of periodic traveling-wave solutions $\bar u^k(k(x-ct))$ described in (H2).
See \cite{NR1,NR2} for a detailed derivation of this kind of nonlinear Whitham's equation in the context of the Saint-Venant and Kuramoto-Sivashinsky equations
when the modulation procedure yields a system rather than an equation.

In our context, we look for the evolution of a localized
perturbation
$k=\ks h_x$ in our co-moving frame,
and so it is enough to retain
the quadratic order approximants
\be\label{mainwhitham}
\ks k_t + \ks q(k)_x= \ks^2d(\ks) k_{xx},
\ee
and
\be\label{mainhj}
\ks h_t+q(\ks h_x)=\ks^2 d(\ks) h_{xx}
\ee
with
$
q(k)=-(\omega_0'(\ks)+c(\ks))k-\frac12 \omega_0''(\ks)k^2\ .
$

For, as is well known, \eqref{whitham} and \eqref{mainwhitham}
are ``asymptotically equivalent''
for such data, in the sense that the difference
between solutions of \eqref{whitham} and \eqref{mainwhitham}
decays faster in all $L^p(\RM)$, $1\le p\le \infty$,
than does the solution itself, which, for an initial perturbation
with nonzero integral, decays at the rate of a heat kernel.
See \cite{DSSS} for a direct derivation of the quadratic Whitham's equation.

The notion of asymptotic equivalence is quantified in the following result.

\begin{lemma}
\label{cl_lemma}
Let $\eta>0$ be arbitrary.
Let $\kappa$ be a solution of \eqref{whitham} with initial datum $\kappa_0$ and $k$ be a solution
of \eqref{mainwhitham} with initial datum $k_0=\kappa_0(\,\cdot\,/\ks)-\ks$,
$E_0:=\|k_0\|_{L^1(\RM)\cap H^3(\RM)}$ sufficiently small.
Then, setting $\tilde\kappa(x,t)=\ks+k(\ks(x-c(\ks)t),\ks t)$,
$$
\big\|(\kappa-\tilde\kappa)(t)\big\|_{L^p(\RM)}\lesssim
E_0^2(1+t)^{-\frac{1}{2}(1-1/p)-\frac12+\eta},
\qquad 1\le p\le \infty;
$$
moreover, if $m_0:=\int_\RM k_0\ne 0$ and $E_1:=E_0+\||\cdot|\,k_0\|_{L^1(\RM)}$ is sufficiently small, then
\be\label{heatconv}
\big\|k(t)-\phi(\cdot,1+t)\big\|_{L^p(\RM)},\quad
\big\|(\kappa-\tilde\kappa)(t)\big\|_{L^p(\RM)}\lesssim
E_0^2(1+t)^{-\frac{1}{2}(1-1/p)-\frac12+\eta},
\qquad 1\le p\le \infty;
\ee
where
\be\label{diff}
\phi(x,t)\ =\ \frac{1}{\sqrt{t}}\ \bar \phi\left(\frac{x+\omega_0'(\ks)t}{\sqrt{t}}\right)
\ee
is the unique
self-similar\footnote{More exactly, self-similar in a frame moving with linear group velocity.}
solution of
\eqref{mainwhitham} determined by $\int \bar \phi=\int k_0$;
in particular,
$
\|\kappa(t)-\ks\|_{L^p(\RM)},\ \|\tilde \kappa(t)-\ks\|_{L^p(\RM)}\gtrsim |m_0|\,(1+t)^{-\frac{1}{2}(1-1/p)} .
$
\end{lemma}

\begin{proof}
See Appendix \ref{cl_proof}.
\end{proof}

Our main result is as follows.

\begin{theorem}\label{main}
Let $\eta>0$ and $K\geq3$.
Under the assumptions of Proposition \ref{oldmain},
let $k$ and $h$ satisfy
the quadratic approximants \eqref{mainwhitham} and \eqref{mainhj} of
the second-order Whitham modulation equations \eqref{whitham} and \eqref{hj} with initial data $k|_{t=0}=\ks\partial_x h_0$, $h|_{t=0}=h_0$,
and let $\psi$ be the phase prescribed in the proof
of Proposition \ref{oldmain} in \cite{JNRZ1} (see \eqref{psidef} below).
Then, for $t>0$,
$2\le p\le\infty$,
\ba\label{refinedest}
\|\tilde u(\cdot- \psi(\cdot,t), t)-\bar u^{\ks(1+ \psi_x(\cdot,t))}(\,\cdot\,)\|_{L^p(\RM)}
&\lesssim E_0
\ln(2+t)\
 (1+t)^{-\frac{3}{4}},\\
\|\ks\partial_x \psi(t)-k(t) \|_{L^p(\RM)}
&\lesssim
E_0 (1+t)^{-\frac{1}{2}(1-1/p)-\frac12+\eta},\\
\|\psi(t) -h(t) \|_{L^p(\RM)}
&\lesssim E_0 (1+t)^{-\frac{1}{2}(1-1/p)+\eta}.\\
\ea
\end{theorem}

\br\label{phase_wavenumber_rmk}
\textup{
Thanks to translation invariance there is no loss in generality in assuming $h_0$ is centered, that is $h_0(-\infty)=-h_0(\infty)$, both in Proposition~\ref{oldmain} and Theorem~\ref{main}. Accordingly, although we will not repeat it explicitly, all (non-localized) functions handled in various lemmas and propositions will be assumed centered. Once this normalization is made, one may recover local phase from local wave number as soon as needed using $\d_x\Psi=\kappa$. It is our will to enforce this important relation that leads us to normalize 
the parametrization according to \eqref{gauge}.
}
\er

\br\label{altconv}
\textup{
From estimates \eqref{refinedest}(i) on $\tilde u$ and
\eqref{mainest}(ii) and \eqref{andpsi} on $\psi_x$ and $\psi$,
we obtain
$$
\|\tilde u(\cdot,t)-\bar u^{\ks(1+\psi_x(\cdot,t))}(\tilde\Psi(\cdot,t))\|_{L^p(\RM)}
\lesssim E_0\ln(2+t)\ (1+t)^{-\frac{3}{4}},\qquad 2\le p\le\infty,
$$
where $\tilde \Psi(\cdot,t)$ is the inverse of $y\mapsto X(y,t):=y-\psi(y,t)$. We insure the existence of such a map by keeping, for any $t$, $\|\psi(t)\|_{L^\infty(\RM)}$ bounded and $\|\psi_x(t)\|_{L^\infty(\RM)}$ small.
Since $\tilde\Psi(x,t)- (x+\psi(x,t))
=\psi(\tilde\Psi(x,t),t)-\psi(\tilde\Psi(x,t)-\psi(\tilde\Psi(x,t),t),t)$
one may translate this into a bound
\be\label{preconv}
\|\tilde u(\cdot,t)-\bar u^{\ks\tilde\Psi_x(\cdot,t)}(\tilde\Psi(\cdot,t))\|_{L^p(\RM)}
\lesssim E_0\ln(2+t)\ (1+t)^{-\frac{3}{4}},\qquad 2\le p\le\infty,
\ee
of the form \eqref{wref},
or degrade it into\footnote{We use here
$\|\tilde\Psi(\cdot,t)-(\,\cdot\,+\psi(\cdot,t))\|_{L^p(\RM)}\lesssim \|\psi(\cdot,t)\|_{L^\infty(\RM)}\|\psi_x(\cdot,t)\|_{L^p(\RM)}$; see \cite{JNRZ2}.}
$
\|\tilde u(\cdot,t)-\bar u^{\ks(1+\psi_x(\cdot,t))}(\cdot+\psi(\cdot,t))\|_{L^p(\RM)}
\lesssim E_0(1+t)^{-\frac12(1-1/p)},
$
obtaining thereby, for $\eta>0$ arbitrary,
and $2\le p\le\infty,$
the exact Whitham comparison
\be\label{ver}
\begin{array}{rcl}
\|\tilde u(\cdot,t)-\bar u^{\ks+k(\cdot,t)}(\cdot+h(\cdot,t))\|_{L^p(\RM)}
\lesssim E_0 (1+t)^{-\frac12(1-1/p)+\eta}.
\end{array}
\ee
Here, we could as well write
\be\label{SSSUversion}
\|\tilde u(\cdot,t)-\bar u(\cdot+h(\cdot,t))\|_{L^p(\RM)}
\lesssim E_0 (1+t)^{-\frac12(1-1/p)+\eta}
\ee
similarly as in \eqref{mainest}, since $k$ is negligible at this
order of approximation.
}
\er

From Theorem \ref{main} and Lemma \ref{cl_lemma}, we see immediately
that the decay rates of Proposition \ref{oldmain} are sharp.
At the same time, we give rigorous validation of the Whitham equation
in two ways.
The first, more obvious way, is to show through \eqref{ver} that asymptotic
behavior consists of modulation by a solution $h$ of the exact Whitham
equations.
Here, as pointed out in Remark \ref{altconv},
modulation in wave number is asymptotically
irrelevant, and only phase shift plays a role.
The second, less direct, but ultimately sharper (by factor
$(1+t)^{-1/4}$ to
$(1+t)^{-1/2}$ in rate of decay) way, is to show through
\eqref{preconv} that {\it a more accurate description of asymptotic behavior
is modulation including both phase shift and variation in wave number}
by a solution $\tilde\Psi$ of an approximate Whitham equation with rapidly
decaying error term (term $\tilde r(t)$ of \eqref{repeat}, together with
terms of similar order coming from the difference $\tilde \Psi_x-
\Psi_x=O(\psi_x^2)$; see Remark \ref{lastrmk}).
See \cite{DSSS} for validation of the Whitham equation in the
altenative
sense of building a family of solutions existing on asymptotically-large but bounded intervals and close to a given asymptotic expansion involving a given solution of the Whitham equation.

\br\label{c0}
\textup{
Similarly as in \cite{JZ1}, all of our analysis goes
through in the general quasilinear $2r$-order parabolic case;
in particular, our results extend to the
(sectorial) Swift--Hohenberg equations treated for localized
perturbations in \cite{S2}.
See \cite{BJNRZ}, Appendix B, for related analysis.
}
\er

\br\label{conjecture}
\textup{
We conjecture that \eqref{refinedest}(i) and thus also
of \eqref{preconv} may be improved for $p>2$ to
$$
\big\|\tilde u(\cdot- \psi(\cdot,t), t)-\bar u^{\ks(1+ \psi_x(\cdot,t))}(\,\cdot\,)\big\|_{L^p(\RM)}
\lesssim E_0
\ln(2+t)\
 (1+t)^{-\frac{1}{2}(1-1/p)-\frac12}
$$
by the use of additional,
$W^{k,p}\to L^p$ estimates outside the scope of this paper;
see Remark \ref{highnorm}.
}
\er

\subsection{Discussion and open problems}\label{discussion}

The Whitham modulation equations, and WKB approximations in general,
are ``magical'' prescriptions of great predictive value
%involving modulations of families of stationary or slowly-varying solutions,
that may be obtained by formal, ``consistency-type,'' considerations
suppressing the sometimes very complicated details of the underlying equations.
As a consequence, they often hide the mechanisms leading to the behavior
they predict.
In particular, rigorous verification of such asymptotic
expansions typically comes from techniques apart from the method of derivation
of the Whitham equation.
Moreover, in the course of verification,
these other techniques may often give additional information
not found in the formal asymptotics.

The present case is no exception.
The Whitham modulation equations, based on slowly varying perturbations,
can be related rather directly at the linearized, spectral level,
to perturbation expansions of small Floquet-number/eigenvalue modes;
see, in various contexts, \cite{DSSS,NR1,NR2}.
Moreover, the same type of critical modes expansion, followed carefully,
leads to the linearized estimates of \cite{JZ1}.
However, to carry this analysis to the nonlinear level involves significant
technical challenges, and appears to require more indirect methods,
motivated by, but at a technical level quite
different from, the formal Whitham modulation.

In \cite{JNRZ1} and the present paper, we have broken the nonlinear analysis
into two distinct pieces, focused respectively on {\it decay} and
{\it asymptotic behavior}.
In \cite{JNRZ1}, by a judiciously chosen nonlinear transformation involving
an implicitly prescribed shift in phase, we showed how to convert rigorously
the picture of behavior afforded by formal asymptotics to a system of integral
equations exhibiting the expected nonlinear decay.
At a technical level, this could be understood as identifying the
main part of the linearized solution operator as a linearized phase shift, and
separating off this part of the behavior by a counterbalancing
nonlinear change in phase.
Put most naively, the analysis is driven by the observation that
the critical mode given by the
kernel of the linearized Bloch operator $L_0$ about the wave
at Bloch frequency $\xi=0$ is
\be\label{naive}
\phi(0,\cdot)=\bar u',
\ee
where $\bar u'$ represents instantaneous translation.

The above analysis gives a simple and self-contained
argument yielding sharp rates of decay to a phase
modulation of the wave.
In the present analysis, we show how to extract from the
system of integral equations derived in \cite{JNRZ1}
an approximate differential equation governing the phase,
and to connect this equation to the formally predicted Whitham equation.
The first step is to go one step further in the decomposition of
the solution,
Taylor expanding the critical mode $\phi(\xi,x)$ of $L_\xi$
about $\xi=0$, and observing
that, under an appropriate normalization (see \eqref{gauge} and \eqref{gauge_eig}),
the first order corrector is
\be\label{naive2}
\partial_\xi \phi(0,\cdot)=i\ks\partial_k u(\cdot;\ks)
\ee
corresponding at linear level to modulation in the wave number $k$;
see Lemma \ref{spectral_connection} below. As often happens for higher-order correctors, the corresponding nonlinear correction can be made in simple, {\it linear} fashion
(see \eqref{pertvar}), since this term is fast enough decaying that nonlinear effects are negligible.
This yields a new residual $z$ decaying at rate \eqref{sharpest}
faster by factor
$(1+t)^{-1/4}$ to $(1+t)^{-1/2}$ than that of the residual $v$ of
\cite{JNRZ1}.
Once this is done, we may, isolating explicit terms $k$
in the $\partial_k\bar u$ direction, discarding as asymptotically
negligible all terms in the integral equations
that are not linear or quadratic in $k$,
and evaluating the resulting
quadratic interaction coefficients, 
obtain
an integral equation that is recognizable 
as the Duhamel (variation of constants) representation of a forced Burgers
equation corresponding to \eqref{mainwhitham}.
From this description and our previous bounds on the residual, we then
readily obtain the convergence result of our main Theorem \ref{main}.

Taken together, these two analyses give a blueprint
for connecting formal asymptotics to the integral equations
natural for analysis.
This should carry over to other interesting settings; indeed, we have
already carried out in \cite{JNRZ2} the extension to the much more complicated
conservation law case.
We note that, despite its difficulty, the analysis in the end, both
for decay and for convergence to the Whithams approximant, is quite short
and easily verified.
In particular, there is
little advance preparation by formal asymptotics, different from
the approach set out in \cite{DSSS}.

We note also that our approach gives ``direct'' access to bounds,
i.e., we ``solve'' rather than ``impose.''
As noted earlier, we obtain as a result somewhat sharper bounds \eqref{preconv}
than what is available from comparison to the exact Whitham equation,
a result that would at least not be easily derived by starting from the
exact Whitham solution to begin with.
%The result is sharper bounds, equation a bit different from (std) Whitham...
That is, the system of integral equations we derive contains
{\it more} information than the formal Whitham approximation.

One interesting problem for future investigation
is the rigorous verification of the
spectral stability conditions in interesting situations
either by analysis or numerical proof.
Another, very interesting direction,
is the derivation of pointwise estimates on nonlocalized
perturbations similarly as has been done for localized solutions
in \cite{J}.

A simplifying aspect of the present analysis is that we were able to
separate the analyses of decay and behavior.
However, there could be an advantage in combining these, in that Burgers
equation is known to decay even for large perturbations, and, thanks to
the maximum principle, decays in $L^\infty$.
Thus, one might hope by such a simultaneous argument structure to treat the
case that $\psi_x$ is large but bounded in $L^1$ but only small in $L^\infty$.
This would be a very interesting extension to carry out.
Finally, as described in \cite{DSSS}, nonlinear stability of
shock-type solutions, for which not only $\psi$ but $k$ approaches different
endstates, is a challenging and very interesting open problem, which
appears to require substantial new ideas beyond those introduced here;
in particular, perturbation around such shock-type solutions and not
a background periodic wave.

\medskip

{\bf Note:}
Similar results have been obtained by different means
by Sandstede and collaborators \cite{SSSU}
using a nonlinear decomposition of phase and amplitude variables
combined with a renormalization iteration process as in \cite{S1,S2}.
Specifically, they obtain the bound \eqref{SSSUversion}
assuming the somewhat stronger localization of initial perturbations of
order $\sim |x|^{-5/2}$ as compared
%TODO: check number of footnote 7 for possible revision...
to the order $\sim |x|^{-1}$ assumed here; see footnote 7, \cite{JNRZ1}.
Our bound \eqref{preconv}, though not of the same form,
gives more precise information by factor
$(1+t)^{-1/4}$ to $(1+t)^{-1/2}$ in the rates of decay;
note also that this takes into account modulation in wave number, whereas
the approximation of \cite{SSSU} does not.
The results of \cite{SSSU} on the other hand
include also convergence (at the same rate)
to a ``nonlinear diffusion wave'' $\phi$ as described in \eqref{diff}.
Combining \eqref{ver} with \eqref{heatconv}, we recover this bound as well,
but with localization $\sim |x|^{-2}$ closer to the $|x|^{-5/2}$ assumption
of \cite{SSSU}.

\section{Preliminaries}\label{s:prelim}

Recall the Bloch solution formula for periodic-coefficient
operators,
\be\label{fullS}
(S(t) g)(x):=(e^{tL}g)(x) = \int_{-\pi}^{\pi} e^{i\xi x} (e^{tL_\xi}  \check g(\xi, \cdot))(x) d\xi,\footnote{In other words,
$\check{(e^{tL}g)}(\xi,x)= (e^{tL_\xi} \check g(\xi,\cdot))(x)$ a consequence of \eqref{defrel}.}
\ee
where
$L_\xi$ is as in \eqref{Lxi},
\be\label{checkg}
\check g(\xi,x):= \sum_{j\in \ZM} \hat g(\xi+2j\pi) e^{i2\pi jx}
\quad
\hbox{\rm (periodic in $x$)}
\ee
denotes the Bloch transform of $g$,
$\hat g(\xi):=\frac{1}{2\pi} \int_\RM e^{-i\xi x}   g(x) dx$
the Fourier transform, and
\be\label{Brep}
g(x)= \int_{-\pi}^{\pi} e^{i\xi x}  \check g(\xi,x) d\xi
\ee
the inverse Bloch transform, or Bloch representation of $g\in L^2(\RM)$.
The generalized Hausdorff--Young inequality
$\|u\|_{L^p(\RM)} \le (2\pi)^{1/p}\| \check u\|_{L^q([-\pi,\pi],L^p([0,1]))}$
for $q\le 2\le p $ and $\frac{1}{p}+\frac{1}{q}=1$ \cite{JZ1},
together with \eqref{Brep}, yields for any $1$-periodic
function $g(\xi,\cdot)$
\be\label{hy}
\Big\|\int_{-\pi}^{\pi}e^{i\xi \cdot} g(\xi,\cdot)d\xi\Big\|_{L^p(\RM)}
\le (2\pi)^{1/p}\|g\|_{L^q([-\pi,\pi],L^p([0,1]))}
\:\; {\rm for} \:\;
q\le 2\le p \:\; {\rm and } \:\; \frac{1}{p}+\frac{1}{q}=1
\ee
where, here and elsewhere, we are denoting
$$
\|g\|_{L^q([-\pi,\pi],L^p([0,1]))}:=\Big(\int_{-\pi}^\pi\|g(\xi,\cdot)\|_{L^p([0,1])}^{q}d\xi\Big)^{1/q}.
$$

By (D3), the zero eigenfunction $\bar u'$ of $L_0$ is simple,
whence by standard perturbation results
bifurcates to an eigenfunction $\phi(\xi,\cdot)$, with associated
left eigenfunction $\tilde \phi(\xi, \cdot)$ and eigenvalue
\be\label{lambda}
\lambda(\xi)= ai \xi - d \xi^2 + O(|\xi|^3),
\ee
where $a$ and $d$ are real and $ d>0$ by assumption (D2) and
complex symmetry $\lambda(\xi)=\bar \lambda(-\xi)$,
each of $\phi$, $\tilde \phi$, $\lambda$ analytic in $\xi$ and defined
for $|\xi|\leq\xi_0$, $\xi_0$ being positive and sufficiently small.

Before refining linear estimates of \cite{JNRZ1}, we need some extra spectral preparation. For this purpose, we set $\tilde \phi(0)=\bar u^{ad}$, 
assume that the
 parametrization is normalized according to \eqref{gauge}, and normalize eigenfunctions in such a way that
\be\label{gauge_eig}
\langle \bar u^{ad},\phi(\xi)\rangle_{L^2([0,1])}\ =\ \langle \tilde\phi(0),\phi(\xi)\rangle_{L^2([0,1])}\ =\ 1
\ee
for all $\xi\in[-\xi_0,\xi_0]$.
A similar preparation is used in \cite{DSSS}; see \cite[Section 4.2]{DSSS}.

\bl[\cite{DSSS}]\label{spectral_connection}
Assuming (H1)-(H2) and (D3), normalize according to \eqref{gauge} and \eqref{gauge_eig}. Then
\be\label{spec_prep1}
\d_\xi\phi(0,\cdot)\ =\ i\ks\d_k u(\cdot;\ks)
\ee
and the constants $a$ and $d$ in \eqref{lambda} are
\be\label{spec_prep2}
a=-\ks c'(\ks),\qquad
d=\ks(1+2\ks\langle u^{ad},\d_k\bar u'(\ks)\rangle_{L^2([0,1])}).
\ee
\el

\begin{proof}
To expand the equation $L_\xi\phi(\xi)=\lambda(\xi)\phi(\xi)$, we split $L_\xi$
$$
\ks L_\xi = L^{(0)} + i\ks \xi L^{(1)} + (i\ks\xi)^2 L^{(2)}.
$$
Then we find
$i\ks a\bar u'=L^{(0)}\d_\xi\phi(0)+i\ks L^{(1)}\bar u'$. Yet differentiating the profile equation with respect to $k$ yields $-\ks c'(\ks)\bar u'=L^{(0)}\d_k\bar u(\ks)+L^{(1)}\bar u'$.
Thus, taking scalar product with $\bar u^{ad}$ leads to
$$
a\ =\ \langle u^{ad},L^{(1)}\bar u'\rangle_{L^2([0,1])}\ =\ -c'(\ks)
$$
hence $L^{(0)}\d_\xi\phi(0)=L^{(0)}\d_k\bar u(\ks)$ and therefore $\d_\xi\phi(0)-\d_k\bar u(\ks)\in \ker L_0=\CM \bar u'$. Normalizations \eqref{gauge} and \eqref{gauge_eig} reduce it to \eqref{spec_prep1}. Afterwards, expanding further, we find
$$
\ks^2 c'(\ks)\d_k\bar u(\ks)-\ks d\ \bar u'\ =\ \frac12 L^{(0)}\d_\xi^2\phi(0)+(i\ks)^2L^{(1)}\d_k\bar u(\ks)+(i\ks)^2L^{(2)}\bar u'.
$$
Taking scalar product with $\bar u^{ad}$ again gives
$$
d\ =\ \ks\langle u^{ad},L^{(1)}\d_k\bar u(\ks)+L^{(2)}\bar u'\rangle_{L^2([0,1])}.
$$
which, thanks to the explicit form
$$
L^{(1)}\d_k\bar u'+L^{(2)}\bar u'\ =\ 2\ks\d_k\bar u(\ks)'+c(\ks)\d_k\bar u(\ks)+\bar u',
$$
completes the proof.
\end{proof}

From now on, although we do not repeat them, we always assume the above normalizations.

\section{Linear estimates}\label{s:linear}

Now, refining slightly the decomposition of \cite{JZ1,JNRZ1},
decompose the solution operator as
\be\label{decomp}
S(t)=R^{\rm p}(t)+ \tilde R(t),
\qquad
R^{\rm p}(t)=(\bar u' + \ks\partial_k \bar u \partial_x) s^{\rm p}(t),
\ee
with
\ba\label{sp}
(s^{\rm p}(t)g)(x)&=\int_{-\pi}^{\pi}
e^{i\xi x}\alpha(\xi) e^{\lambda(\xi)t} \langle \tilde \phi(\xi,\cdot), \check g(\xi,\cdot)\rangle_{L^2([0,1])} d\xi,
\ea
and
\ba\label{tildeS}
(\tilde R(t) g)(x)&:=
\int_{-\pi}^{\pi} e^{i\xi x} (1-\alpha(\xi))
(e^{L_\xi t}  \check g(\xi))(x) d\xi
+ \int_{-\pi}^{\pi} e^{i\xi x} \alpha(\xi)
(e^{L_\xi t} \tilde \Pi(\xi) \check g(\xi))(x) d\xi\\
&
+ \int_{-\pi}^{\pi}
e^{i\xi x}\alpha(\xi) e^{\lambda(\xi) t}(\phi(\xi,x)-\phi(0,x)-\xi\partial_\xi\phi(0,x))
\langle \tilde \phi(\xi),\check g(\xi) \rangle_{L^2([0,1])} d\xi
,
\ea
where $\alpha$ is a smooth cutoff function such that $\alpha(\xi)=0$ for
$|\xi|\ge \xi_0$ and $\alpha(\xi)=1$ for $|\xi|\le \frac{1}{2}\xi_0$,
\be\label{Pi}
\Pi^{\rm p}(\xi):=
\phi(\xi)\langle \tilde \phi(\xi), \cdot\rangle_{L^2([0,1])}
\ee
denotes the eigenprojection onto the eigenspace $\Range \{\phi(\xi)\}$
bifurcating from $\Range \{ \bar u'\}$ at $\xi=0$, $\tilde \phi$ the associated left
eigenfunction, and $\tilde \Pi:=\Id-\Pi^{\rm p}$,
each well-defined on $\suppt(\alpha)\subset[-\xi_0,\xi_0]$.

\bpr[\cite{JZ1,JNRZ1}]\label{greenbds}
Under assumptions (H1)-(H2) and (D1)-(D3),
for all $t>0$, $2\leq p\leq \infty$,
\begin{align}\label{finale}
\left\|
\partial_x^l\partial_t^m s^{\rm p}(t) \partial_x^r g \right\|_{L^p(\RM)}
\lesssim
\min \begin{cases}
(1+t)^{-\frac{1}{2}(1-1/p)-\frac{l+m}{2}}\|g\|_{L^1(\RM)},\\
(1+t)^{-\frac{1}{2}(1/2-1/p)-\frac{l+m}{2}}
\|g\|_{L^2(\RM)},
\end{cases}
\end{align}
for
$0\leq r\leq K+1$,
and for some $\eta>0$
and $0\leq l+2m,r \leq K+1$,

\begin{align}\label{finalg}
\left\|\partial_x^l \partial_t^m \tilde R(t) \partial_x^r g\right\|_{L^p(\RM)}
&\lesssim
\min
\begin{cases}
(1+t)^{-\frac{1}{2}(1-1/p)-\frac{1}{2}}\| g\|_{L^1(\RM)\cap H^{l+2m+1}(\RM)},\\
e^{-\eta t}\|\partial_x^r g\|_{H^{l+2m+1}(\RM)}+(1+t)^{-\frac{1}{2}(1/2-1/p)-1}\|g\|_{L^2(\RM)},\\
\end{cases}
\end{align}
\epr

\begin{proof}
The estimates on $s^{\rm p}$ were proved in \cite{JNRZ1}. Using \eqref{hy} together with spectral projection/direct computation for low
Floquet numbers
and standard semigroup estimates for high Floquet numbers,
estimates on $\tilde R$ are proved exactly as were the estimates on $\tilde S$ in \cite{JNRZ1}, with the observation
that the substitution of factor
$ (\phi(\xi)-\phi(0)-\xi \partial_\xi \phi(0))$ in $\tilde R$ for factor
$ (\phi(\xi)-\phi(0))$ in $\tilde S$ introduces an additional factor
of $|\xi|$ in the estimates, hence an additional $(1+t)^{-1/2}$ factor of decay.
As we will build on the main arguments of these proofs to establish further linear estimates, we briefly recall them here.

{\it (i) (Proof of \eqref{finale}).}
In the case $l=m=r=0$ estimates on $s^{\rm p}$ 
follows by choosing successively $s=1$ and $s=2$, introducing $s'$ such that $1/s+1/s'=1$,
and observing that
$$
\begin{aligned}
\displaystyle
\Big\|\ x&\mapsto\int_{-\pi}^{\pi}
e^{i\xi x}\alpha(\xi) e^{\lambda(\xi) t}
\langle \tilde \phi(\xi),\check g(\xi) \rangle_{L^2([0,1])}d\xi\ \Big\|_{L^p(\RM)}\\
&\displaystyle\lesssim
\|\ (\xi,x)\mapsto\alpha(\xi) e^{\lambda(\xi) t} |\langle \tilde \phi(\xi),\check g(\xi) \rangle_{L^2([0,1])}|\ \|_{L^q([-\pi,\pi],L^p([0,1]))}\\
&\displaystyle\lesssim
\|\xi\mapsto e^{-\eta \xi^2 t}\|_{L^{r_{s,p}}([-\pi,\pi])}\ \|\xi\mapsto\alpha(\xi)^{1/2}|\langle \tilde \phi(\xi),\check g(\xi) \rangle_{L^2([0,1])}|\|_{L^{s'}([-\pi,\pi])}\\
&\displaystyle\lesssim
(1+t)^{-\frac{1}{2}(1/s-1/p)}\|\xi\mapsto\alpha(\xi)^{1/2}|\langle \tilde \phi(\xi),\check g(\xi) \rangle_{L^2([0,1])}|\|_{L^{s'}([-\pi,\pi])},
\end{aligned}
$$
where $1/p+1/q=1$ and $1/s'+1/r_{s,p}=1/q$, so that  $1/r_{s,p}=1/s-1/p$. Here
we have used the fact that, by (D2), for some $\eta>0$, $|e^{\lambda(\xi)t}\alpha^{1/2}(\xi)|
\le e^{-\eta \xi^2 t}$.

For $s=2$, we combine this with the Cauchy-Schwarz inequality and Parseval identity to obtain
$$
\begin{aligned}
\Big\|\xi\mapsto\alpha(\xi)^{1/2}|\langle \tilde \phi(\xi),\check g(\xi)\rangle_{L^2([0,1])}\Big\|_{L^{2}([-\pi,\pi])}
&\leq \sup_{|\xi|\leq\xi_0}\|\tilde \phi(\xi,\cdot)\|_{L^2([0,1])}\ \|\check g\|_{L^2([-\pi,\pi],L^2([0,1])}\\
&\lesssim \sup_{|\xi|\leq\xi_0}\|\tilde \phi(\xi,\cdot)\|_{L^2([0,1])}\ \|g\|_{L^2(\RM)},
\end{aligned}
$$
while for $s=1$ we expand
$
\langle \tilde \phi(\xi),\check g(\xi)\rangle_{L^2([0,1])}=\sum_{j\in\ZM}\hat{\tilde \phi  }_{j}(\xi)^* \hat g(\xi+2j\pi),
$
where
$\hat{\tilde \phi }_{j}(\xi)$ denotes the $j$th Fourier coefficient in the
Fourier expansion of periodic function $\tilde \phi(\xi,\cdot)$, and $z^*=\bar z$ denotes complex conjugate, and apply Hausdorff-Young's inequality, $\|\hat g\|_{L^\infty(\RM)}\le \|g\|_{L^1(\RM)}$ with
Cauchy--Schwarz' inequality,
\[
\alpha^{1/2}(\xi)\sum_j |\hat{\tilde  \phi}_j(\xi)|\le
\alpha^{1/2}(\xi)\sqrt{\sum_j (1+|j|^2)|\hat {\tilde \phi}_j(\xi)|^2\sum_j (1+|j|^{-2})}
\le C\alpha^{1/2}(\xi) \|\tilde \phi(\xi)\|_{H^1([0,1])},
\]
to get
$$
\begin{aligned}
\|\xi\mapsto\alpha(\xi)^{1/2}|\langle \tilde \phi(\xi),\check g(\xi)\rangle_{L^2([0,1])}|\|_{L^{\infty}([-\pi,\pi])}
&\lesssim \sup_{|\xi|\leq\xi_0}\|\tilde \phi(\xi,\cdot)\|_{H^1([0,1])}\ \|g\|_{L^1(\RM)}.
\end{aligned}
$$

This yields the result for $l=m=r=0$. Estimates for general $l,m,r \ge 0$ go similarly, passing $\partial_x^r$ derivatives onto $\tilde \phi(\xi)$ in the inner product using integration by parts and noting that $\partial_x^l$ and $\partial_t^m$ derivatives bring down bounded factors $(i\xi)^l$ and $\lambda(\xi)^m$ enhancing decay through
$$
\|\xi\mapsto |\xi|^{l+m}e^{-\eta \xi^2 t}\|_{L^{r_{s,p}}([-\pi,\pi])}\lesssim (1+t)^{-\frac{1}{2}(1/s-1/p)-\frac{l+m}{2}}
$$

\medskip
{\it (ii) (Proof of \eqref{finalg}).}
By (D1)--(D2), benefiting from standard parabolic resolvent estimates \cite{He} together with the fact that (by basic ODE regularity theory) $H^{l+1}([0,1])$ and $L^2([0,1])$ spectra coincide, we may apply Pr\"uss' Theorem \cite{Pr} and get
$$
|e^{L_\xi t}(1-\alpha(\xi))|_{H^{l+1}([0,1])\to H^{l+1}([0,1])},\;\;
|\alpha(\xi) e^{L_\xi t}\tilde \Pi(\xi)|_{H^{l+1}([0,1])\to H^{l+1}([0,1])}
\lesssim e^{-\eta t},
\quad \eta>0,
$$
whence, by Sobolev embedding,
$$
|\partial_x^l e^{L_\xi t}(1-\alpha(\xi))|_{H^{l+1}([0,1])\to L^p([0,1])},\;\;
|\partial_x^j \alpha(\xi) e^{L_\xi t}\tilde \Pi(\xi)|_{H^{l+1}([0,1])\to L^p([0,1])}
\lesssim e^{-\eta t}
$$
for $2\le p\le \infty$. The $W^{l,p}(\RM)$ norms of the first two terms of \eqref{tildeS}, by \eqref{hy} and
Parseval's identity,
$$
\frac{1}{2\pi}\|g\|_{H^{l+1}(\RM)}^2
=\|\check g\|_{L^2([-\pi,\pi],L^2([0,1]))}^2+\big\|(\partial_x +i\xi)^{l+1}\check g(\xi)\big\|_{L^2([-\pi,\pi],L^2([0,1]))}^2,
$$
are thus bounded by $Ce^{-\eta t}\|g\|_{H^{l+1}(\RM)}$. The $W^{l,p}(\RM)$ norm of the third term may be bounded similarly as in the estimation of $s^{\rm p}$ above, noting that the factor
$(\phi(\xi)-\phi(0)-\xi\d_\xi\phi(0))=O(\xi^2)$ introduces an additional factor of $(1+t)^{-1}$ decay.
This establishes the result for $m=r=0$; other cases go similarly,
noting that $\partial_t e^{L_\xi t} \tilde \Pi(\xi)= L_\xi e^{L_\xi t} \tilde \Pi(\xi)$,
with $L_\xi$ a second-order operator, so that we may essentially trade one $t$-derivative for two $x$-derivatives.
\end{proof}

\bpr[\cite{JNRZ1}]\label{modprop}
Under assumptions (H1)--(H2) and (D1)--(D3),
for all $t>0$ and $2\leq p\leq \infty$,
\be\label{Spmod}
\| \partial_x^l\partial_t^m s^{\rm p} (t)  (h_0\bar u')\|_{L^p (\RM)}
\lesssim
(1+t)^{-\frac{1}{2}(1-1/p)+\frac{1}{2}-\frac{l+m}{2}}
\|\partial_x h_0\|_{L^1 (\RM)},
\ee
when $l+m\ge 1$ or else $l=m=0$ and $p=\infty$,
and, for $0\leq l+2m \leq K+1$,
\be\label{tildeSmod}
\|\partial_x^l \partial_t^m \tilde R (t)(h_0\bar u')\|_{L^p(\RM)}
\lesssim (1+t)^{-\frac{1}{2}(1-1/p)-1}
 \|\partial_x  h_0\|_{L^1 (\RM) \cap H^{l+2m+1} (\RM)
},
\ee
and when $t\le 1$
\be\label{spdiff}
\begin{array}{rcl}
\|\partial_x^l \partial_t^m (R^{\rm p} (t) -\Id) (h_0\bar u')\|_{L^p (\RM) }
&\lesssim&\|\partial_x h_0\|_{L^1 (\RM)\cap H^{l+2m+1}(\RM)},\\
\|\partial_x^l \partial_t^m (s^{\rm p} (t)(h_0\bar u')- h_0)\|_{L^p (\RM)}
&\lesssim&
\|\partial_x h_0\|_{L^1 (\RM)\cap L^{2}(\RM)}.
\end{array}
\ee
\epr

\begin{proof}
Again, estimates on $s^{\rm p}$ were proved in \cite{JNRZ1} whereas the proof of those on $R^{\rm p}$ and $\tilde R$ goes exactly as the proof of the corresponding estimates for $S^{\rm p}$ and $\tilde S$ in \cite{JNRZ1},
noting that the substitution of factor
$ (\phi(\xi)-\phi(0)-\xi \partial_\xi \phi(0))$ in $\tilde R$ for factor
$ (\phi(\xi)-\phi(0))$ in $\tilde S$ introduces an additional factor
of $(1+t)^{-1/2}$ in the decay rate for $\tilde R$.
% For completeness, we
We briefly outline the
%main arguments.
new points of the arguments.

{\it (i) (Proof of \eqref{Spmod}).}
The inequality \eqref{Spmod} was established in \cite{JNRZ1},
Proposition 4.1.
\medskip

{\it (ii) (Proof of \eqref{tildeSmod}).}
Contribution of the last term in \eqref{tildeS} is bounded as in the proof of \eqref{Spmod}. For the remaining terms, we first split
$$
(h_0\bar{u} ')\,\check{}\ (\xi,x)=\bar u'(x)\check{h_0}(\xi,x)
=\sum_{j\in\ZM}\bar u'(x) e^{2\pi j x}\widehat{h_0}(2\pi j+\xi)
$$
then apply to each term of the sum the semigroup bounds as above and, after noting that $(1-\alpha(\xi))\lesssim \xi$ and that
$$
\tilde \Pi(\xi) \bar u'=
\tilde \Pi(\xi)\Big[(\tilde \Pi(\xi)-\tilde \Pi(0)) \bar u' \big]
\quad\hbox{\rm with}\quad
|\tilde \Pi(\xi)-\tilde \Pi(0)|_{H^{l+2m+1}([0,1])\to H^{l+2m+1}([0,1])}
\lesssim |\xi|,
$$
achieve the proof with the Cauchy--Schwarz inequality
$$
\begin{aligned}
\Big\|x\mapsto\bar u'(x)\widehat{\d_x h_0}(\xi) \Big\|_{ H^{l+2m+1}([0,1])}
&+\sum_{j\neq0}\Big\|x\mapsto\frac{\bar u'(x) e^{2i\pi jx}}{i(\xi+2\pi j)}
\widehat{\d_x h_0}(\xi+2\pi j) \Big\|_{ H^{l+2m+1}([0,1])}\\
&\lesssim |\widehat{\d_x h_0}(\xi)|+
\sum_{j\neq0} \left|\frac{(2\pi j)^{l+2m+1}
\widehat{(\d_x^{l+2m+2} h_0)}(\xi+2\pi j)}{(\xi+2\pi j)^{l+2m+2}}\right|\\
&\lesssim |\widehat{\d_x h_0}(\xi)|+
\sqrt{ \sum_j \frac{1}{(1+|j|)^2} \sum_{j'}|\widehat{(\d_x^{l+2m+2} h_0)}(\xi+2\pi j')|^2}\\
&\lesssim |\widehat{\d_x h_0}(\xi)|+\sqrt{\sum_{j\in\ZM}|\widehat{(\d_x^{l+2m+2} h_0)}(\xi+2\pi j)|^2},
\end{aligned}
$$
the $L^2([-\pi,\pi])$ norm of the last quantity being bounded by $\|\d_x h_0\|_{H^{l+2m+1}(\RM)}$.

\medskip
{\it (iii) (Proof of \eqref{spdiff}).}
Expanding
$
R^{\rm p}(t)-\Id=(R^{\rm p}(t)-R^{\rm p}(0)) -\tilde R(0)
=t\partial_t R^{\rm p}(s(t)) -\tilde R(0)
$
for some $0<s(t)<t$, we obtain the first inequality by combining \eqref{Spmod}
and \eqref{tildeSmod}.
The second inequality was established in \cite{JNRZ1}, Proposition 4.1.
\end{proof}

We require also the following key new estimates, proved by similar techniques.
The first statement reflects the fact that modulations approximately travel at reference group speed $a$. The second one quantifies the fact that the (parabolic) second-order linearized modulation equation is given by \eqref{heat}.

\bl\label{cancel}
Under assumptions (H1)-(H2) and (D1)-(D3),
for all $t>0$, $2\leq p\leq \infty$,
\begin{align}\label{e:cancel}
\left\|(\partial_t-a\partial_x)(s^{\rm p}(t) g)\right\|_{L^p(\RM)}
\lesssim
\min
\begin{cases}
(1+t)^{-\frac{1}{2}(1-1/p)-1}\|g\|_{L^1(\RM)},\\
(1+t)^{-\frac{1}{2}(1/2-1/p)-1}\|g\|_{L^2(\RM)},\\
\end{cases}
\end{align}
\be\label{dSpmod}
\| (\partial_t - a\partial_x) (s^{\rm p}(t)(h_0\bar u'))\|_{L^p(\RM)}
\lesssim(1+t)^{-\frac{1}{2}(1-1/p)-\frac{1}{2}}
\|\partial_x h_0\|_{L^1(\RM)}.
\ee
\el

\begin{proof}
Differentiating \eqref{sp}, we have
\be\label{dsp}
(\partial_t-a\partial_x)(s^{\rm p}(t)g)(x)
=\int_{-\pi}^{\pi} (\lambda(\xi)- ai\xi)e^{i\xi x}\alpha(\xi) e^{\lambda(\xi)t}
\langle \tilde \phi(\xi), \check g(\xi)\rangle_{L^2([0,1])}(\xi) d\xi
\ee
with $\lambda(\xi)-ia\xi=O(|\xi|^2) $, whence the result follows, again by \eqref{hy}.
\end{proof}

\bpr\label{heatcomp}
Assuming (H1)--(H2) and (D1)--(D3),
let $\sigma (t)$ be the solution operator of the convected
heat equation
\be\label{heat}
u_t=au_x+d u_{xx},
\ee
where $a$, $d$ are as in \eqref{lambda},
and let $g$ be a periodic function on $[0,1]$, $g\in L^2([0,1])$.
Then, for all $t>0$, $2\leq p\leq \infty$,
and $l\in\NM$,
\ba\label{heatspdiffx}
\| \partial_x^l s^{\rm p}(t)(h_0 g)-&\langle \tilde \phi(0),g\rangle\ \sigma(t)(\partial_x^l h_0)\|_{L^p(\RM)}\\
&\lesssim
\begin{cases}
(1+t)^{-\frac{1}{2}(1-1/p)-\frac12}t^{-\frac{l-1}{2}} \|\partial_x h_0\|_{L^1(\RM)\cap L^\infty(\RM)}&\quad l\geq1;\\
(1+t)^{-\frac{1}{2}(1-1/p)-\frac12}\|\partial_x h_0\|_{L^1(\RM)\cap L^2(\RM)}&\quad l=0.
\end{cases}
\ea
\epr

\begin{proof}
Expressing
$$
\begin{array}{rcl}
\displaystyle
\langle \tilde \phi(0),g\rangle_{L^2([0,1])}\, (\sigma(t)\partial_x^l h_0)(x)&=&
\displaystyle
\int_\RM(i\xi)^le^{i\xi x} e^{(ia\xi -d\xi^2)t} \langle \tilde \phi(0),g\rangle_{L^2([0,1])}\,
\widehat{h_0} (\xi) d\xi\\[1em]
\displaystyle
\partial_x^l s^{\rm p}(t)(gh_0)(x)&=&
\displaystyle
\int_{-\pi}^{\pi} (i\xi)^l e^{i\xi x}\alpha(\xi) e^{\lambda(\xi)t}
\sum_j \langle \tilde \phi(\xi),g e^{i2\pi j \cdot }\rangle_{L^2([0,1])}
\widehat{h_0} (\xi+2j\pi) d\xi
\end{array}
$$
and subtracting, we obtain, for $l\in\NM$,
$$
\begin{aligned}
\big[\langle \tilde \phi(0),g\rangle \sigma(t)(\d_x^l h_0)&-\d_x^l(s^{\rm p}(t)(gh_0))\big](x)
=
\int_{\RM\setminus[-\pi,\pi]}
e^{i\xi x} e^{(ia\xi -d\xi^2)t}
(i\xi)^{l-1}
\langle \tilde \phi(0),g\rangle\,\widehat{\d_x h_0}(\xi)d\xi\\
&\quad+
\int_{-\pi}^{\pi}
e^{i\xi x}(1-\alpha(\xi)) e^{(ia\xi -d\xi^2)t}(i\xi)^{l-1}
\langle \tilde \phi(0),g\rangle\,\widehat{\d_x h_0}(\xi) d\xi\\
&\quad+
\int_{-\pi}^{\pi}
e^{i\xi x}\alpha(\xi)(i\xi)^{l-1} (e^{(ia\xi-d\xi^2)t}-e^{\lambda(\xi)t})
\langle \tilde \phi(0),g\rangle\,\widehat{\d_x h_0}(\xi) d\xi\\
&\quad+
\int_{-\pi}^{\pi}
e^{i\xi x}\alpha(\xi) e^{\lambda(\xi)t} (i\xi)^{l-1}
\langle \tilde \phi(0)-\tilde \phi(\xi),g\rangle\,\widehat{\d_x h_0} (\xi) d\xi\\
&\quad-
\int_{-\pi}^{\pi}
e^{i\xi x}\alpha(\xi) e^{\lambda(\xi)t} (i\xi)^{l}
\sum_{j\ne 0} (\widehat{ \tilde \phi(\xi)g})_j ^*
\left(\frac{1}{i(\xi+2j\pi)}\right) \widehat{\d_x h_0} (\xi+2j\pi) d\xi.
\end{aligned}
$$
where $(\widehat{ \tilde \phi(\xi)g})_j$  denotes the $j$th coefficient in the Fourier expansion of periodic function $\tilde \phi(\xi)g$. 
When $l=0$, the first term is bounded in $L^p(\RM)$ by $Ce^{-\eta t}\|\d_xh_0\|_{L^{1/(1/p+1/2)}(\RM)}$ using
$$
\left\|\xi\mapsto e^{-d\xi^2 t}\xi^{-1}\right\|_{L^2(\RM\setminus[-\pi,\pi])}\lesssim e^{-\eta t}.
$$
When $l\geq1$, it is bounded by $Ct^{-\frac{l-1}{2}}e^{-\eta t}\|\d_xh_0\|_{L^p(\RM)}$ as the convolution of $\d_xh_0$ with a kernel that is bounded pointwisely (using Haussdorff-Young estimates) by
$$
x\ \longmapsto\ C\ t^{-\frac{l}{2}}e^{-\eta\,t}\frac{1}{1+\frac{x^2}{t}}
$$
and therefore is bounded in $L^1(\RM)$ by $Ct^{-\frac{l-1}{2}}e^{-\eta\,t}$, for some $\eta>0$.
Now, using $(1-\alpha(\xi))\lesssim |\xi|$,
$$
\|\tilde \phi(\xi)-\tilde \phi(0)\|_{L^2([0,1])}\leq |\xi| \sup_{|\xi'|\leq\xi_0}\|\d_\xi\tilde\phi(\xi')\|_{L^2([0,1])},
$$
$\lambda(\xi)-(ia\xi-d\xi^2)=O(|\xi|^3)$ to get
$$
|e^{(ia\xi-d\xi^2)t}-e^{\lambda(\xi)t}|\ =\
e^{-d\xi^2t}|1-e^{(\lambda(\xi)-(ia\xi-d\xi^2))t}|\
\lesssim\ |\xi|^3 e^{-d\xi^2t}
$$
and by Cauchy-Schwarz estimate
$$
\begin{aligned}
\sup_{|\xi|\leq\xi_0}\left|\sum_{j\ne 0} (\widehat{ \tilde \phi(\xi)g})_j ^*
\left(\frac{1}{i(\xi+2j\pi)}\right) \widehat{\d_x h_0} (\xi+2j\pi)\right|
&\lesssim
\| \partial_x h_0\|_{L^1(\RM)}
\sup_{|\xi|\leq\xi_0}\sum_{j\ne 0} \frac{|\widehat{(\tilde \phi(\xi)g)}_j|}{(1+|j|)}\\
&\lesssim
\|\partial_x h_0\|_{L^1(\RM)}
\sup_{|\xi|\leq\xi_0}\|\tilde \phi(\xi) g\|_{L^2([0,1])} \\
&\lesssim
\|g\|_{L^2([0,1])}
\sup_{|\xi|\leq\xi_0}\|\tilde \phi(\xi)\|_{L^{\infty}([0,1])} \|\partial_x h_0\|_{L^1(\RM)},
\end{aligned}
$$
we may apply \eqref{hy} to the other terms and bound them in $L^p(\RM)$ by
$$
\|\xi\mapsto |\xi|^{l}e^{-\eta\xi^2t}\|_{L^q([-\pi,\pi])}\ \|\d_xh_0\|_{L^1(\RM)}
\lesssim (1+t)^{-\frac{1}{2}(1-1/p)-\frac{l}{2}}\ \|\d_xh_0\|_{L^1(\RM)},
$$
with $q$ such that $1/p+1/q=1$ and $\eta>0$.
\end{proof}

\section{Nonlinear stability estimates}\label{s:pert}

\subsection{Perturbation equations and nonlinear decomposition}
We now refine the nonlinear perturbation equations of \cite{JNRZ1}.
First, recall the iteration scheme introduced in \cite{JNRZ1}.
For $\tilde{u}$ satisfying $\ks u_t=\ks^2u_{xx}+f(u)+\ks cu_x$,
we introduced the nonlinear perturbation
\be\label{veq}
 v(x,t)= \tilde{u}(x-\psi(x,t),t)-\bar{u}(x)
\ee
and showed that it satisfies
\be\label{vcancel}
\ks\left(\partial_t-L\right)(v+\psi\bar u_x)=
\ks\mathcal{N}:=\cQ+ \cR_x +(\ks\d_t+\ks^2\d_x^2)S+T,
\ee
where
\ba\label{eqn:QRST}
\cQ:=f(v+\bar{u})-f(\bar{u})-df(\bar{u})v,
&\qquad
\cR:=-\ks v\psi_t -\ks^2 v\psi_{xx}+ \ks^2(\bar u_x +v_x)\frac{\psi_x^2}{1-\psi_x},\\
\cS:= v\psi_x ,
\quad \hbox{\rm and }
&\qquad
\cT:=-\left(f(v+\bar{u})-f(\bar{u})\right)\psi_x .
\ea

Defining the phase
$\psi$ implicitly by
\ba\label{psidef}
\psi(t)&=s^{\rm p}(t)(d_0+h_0\bar u')+\int_0^t s^{\rm p}(t-s)\N(s)ds
\\
&-(1-\chi(t))\Big(s^{\rm p}(t)(d_0+h_0\bar u')-h_0+\int_0^t s^{\rm p}(t-s)\N(s)ds\Big),
\ea
where $ d_0:=\tilde u_0(\cdot-h_0(\cdot))-\bar u$ and $\chi$ is
a smooth cutoff such that $\chi(t)$ is zero for $t\le 1/2$
and one for $t\ge 1$, applying Duhamel's principle (variation
of constants), and rearranging terms, we then obtained
the integral representation
\ba\label{oldnl}
v(t)&=\tilde S(t) (d_0+h_0\bar u')+ \int_0^t \tilde S(t-s) \N(s) ds\\
&\quad
+(1-\chi(t))\Big(\bar u' s^{\rm p}(t)(d_0+h_0\bar u')-h_0\bar u'+\int_0^t \bar u' s^{\rm p}(t-s)\N(s)ds\Big),\\
\ea
closing the system for $(v,\psi)$, where $\tilde S(t):=S(t)-\bar u' s^{\rm p}(t)$.

As described in \cite{JNRZ1}, the choice \eqref{psidef} is designed to
cancel, modulo an ``initial layer'' term $(1-\chi)(\dots)$ vanishing
for $t\ge 1$, any $\bar u' s^{\rm p}$ terms that would otherwise occur
in the description of $v$, that is, to extract as much as possible
the expected phase shift from the solution $\tilde u$.
Indeed, it is essentially uniquely determined by this requirement together
with the requirement that $\psi(\cdot,0)=h_0$.
The crucial fact making possible the estimates of Proposition \ref{oldmain}
is that the nonlinear term $\mathcal{N}$ is of quadratic order in terms of $v$
and derivatives of $\psi$ decaying at the rate of a heat kernel,
whereas propagators $\tilde S$ and $\nabla_{x,t} s^{\rm p}$ decay at the
rate of a differentiated heat kernel, leading to a closable iteration
in variables $(\nabla_{x,t}\psi, v)$ and derivatives.
This is analogous to the situation of the standard iteration scheme
used to show decay with respect to localized ($L^1$) data of
a Burgers equation $k_t-k_{xx}= (\frac12 k^2)_x$ treated as a perturbation
of the heat equation, which is perhaps natural in view of the
expected asymptotics \eqref{mainwhitham}.
We refer the reader to \cite{JNRZ1} for further details;
see also \cite{Z1,HoZ} for related arguments in the context of viscous shock stability.

Recall from the introduction that $\psi_x$ approximates perturbation
in wave number according to the formal Whitham approximation.
Accordingly, refining decomposition \eqref{veq} with 
Lemma \ref{spectral_connection} in mind,
we define the new nonlinear perturbation variable
\be\label{pertvar}
z(x,t) :=\tilde{u}(x-\psi(x,t),t)-\bar{u}(x)-\d_k\bar u(x)\ks\d_x\psi(x,t)=v(x,t)-\d_k\bar u(x)\ks\d_x\psi(x,t),
\ee
approximately separating out expected modulation in the wave-number
along with the phase shift.

\begin{lemma}\label{lem:canest}
For $\psi$ defined as in \eqref{psidef},
the nonlinear residual $z$ defined in \eqref{pertvar} satisfies
\ba\label{closed}
z(t)&=\tilde R(t) (d_0+h_0\bar u')+ \int_0^t \tilde R(t-s) \N(s) ds\\
&\quad
+(1-\chi(t))\Big(R^{\rm p}(t)(d_0+h_0\bar u')-h_0\bar u'-\d_k\bar u\,\ks\d_x h_0+\int_0^t R^{\rm p}(t-s)\N(s)ds\big).\\
\ea
\end{lemma}

\begin{proof}
Using
$z(t)=v(t)-\d_k \bar u\,\ks\psi_x(t)$, $\tilde R(t)=\tilde S(t)-\d_k\bar u\,\ks\d_x s^{\rm p}(t)$ and $R^{\rm p}(t)=(\bar u'+\d_k\bar u\,\ks \d_x)s^{\rm p}(t)$, equation \eqref{closed} follows immediately
from \eqref{psidef} and \eqref{oldnl}.
\end{proof}

\subsection{Refined stability estimate}\label{s:proof}
Comparing \eqref{oldnl} and \eqref{closed}, we see that we have
by these manipulations effectively exchanged for
$\tilde S$ and $s^{\rm p}$ the faster-decaying
propagators $\tilde R$ and $R^{\rm p}$ in the representations of $z$ vs. $v$.
With these improvements, we may refine Proposition \ref{oldmain} as follows.

\begin{proposition}\label{step}
Under the assumptions of Proposition \ref{oldmain},
for $t>0$,
$2\le p\le \infty$,
we have the sharpened estimate
\ba\label{sharpest}
\|z(t)\|_{L^p(\RM)} &\lesssim E_0 \ln(2+t)\ (1+t)^{-\frac{3}{4}}.
\ea
\end{proposition}

\begin{proof}
From the basic bounds \eqref{mainest} of Proposition \ref{oldmain}, we have
$
\|\N(t)\|_{H^1(\R)}
\le CE_0 (1+t)^{-\frac{3}{4}}.
%\\
$
Applying the bounds of Propositions \ref{greenbds} and \ref{modprop}
to system \eqref{closed}, we obtain for any $2\le p \le \infty$
\ba\label{sest}
\|z (t) \|_{L^p(\RM)}& \le
CE_0(1+t)^{-\frac{1}{2}(1-1/p)-\frac{1}{2}} + CE_0 \int_{0}^{t} (1+t-s)^{-\frac{1}{2}(1/2-1/p)-1}(1+s)^{-\frac{3}{4}}ds\\
&
\le CE_0 \ln(2+t)\ (1+t)^{- \frac{3}{4}  }.
\ea
\end{proof}

\br\label{highnorm}
\textup{
We expect that it should be possible to improve \eqref{sharpest} to
$$
 \|z(t)\|_{L^p(\RM)} \lesssim
E_0 \ln(2+t)\ (1+t)^{-\frac{1}{2} (1-1/p)-\frac12},
\qquad
2\le p\le \infty,
$$
by substituting in \eqref{sest}
$W^{k,p}$ bounds on $\mathcal{N}$ and $W^{k,p}\to L^p$ bounds
on the solution operators $R^{\rm p}$ and $\tilde R$.
However, to obtain such $W^{k,p}\to L^p$ bounds for $p> 2$
would appear to require
%misleading, removed:
%low-norm
techniques outside the Hausdorff--Young-type estimates used in this
paper, perhaps pointwise bounds as in \cite{J},
%TODO: add this? or not? see HoZ2 for related...
or (operator-valued) Fourier multiplier techniques as suggested in Remark 4.2.2, \cite{JNRZ1}.
}
\er

\section{The Whitham equation}
We complete our investigation by connecting
to the Whitham modulation equation.

\bl
Assuming (H1)--(H2), (D1)--(D3),
$\ks\mathcal{N}(t)= f^{\rm p}\ \ks^2\psi_x(t)^2 + r(t)$, where
\be\label{fsharp}
f^{\rm p}=\frac12d^2f(\bar u)(\d_k\bar u,\d_k\bar u)+\ks\d_k\bar u''-\frac{1}{\ks} df(\bar u)\d_k\bar u+\bar u''-a\d_k\bar u'
\ee
is periodic and $\|r(t)\|_{L^1(\RM)}\lesssim E_0^2(1+t)^{-1}$.
\el

\begin{proof}
Immediate from Lemma \ref{lem:canest},
\eqref{mainest},
\eqref{sharpest}, and
$
\|\psi_t(t)-a\psi_x(t)\|_{L^2(\RM)}=O((1+t)^{-3/4}) ,
$
a consequence of Lemma \ref{cancel}.
\end{proof}

\bl\label{quadcoeff}
Setting $e:=\langle \tilde \phi(0), f^{\rm p}\rangle_{L^2([0,1])}$, we have $e=\frac12\omega_0''(\ks)$.
\el

\begin{proof}
We may rewrite
$$
f^{\rm p}=-L\d_k\bar u+2\ks\d_k\bar u''-\omega_0'(\ks)\d_k\bar u'+\bar u''+\frac12d^2f(\bar u)(\d_k\bar u,\d_k\bar u)
$$
thus
$$
e=\langle\tilde\phi(0),2\ks\d_k\bar u''-\omega_0'(\ks)\d_k\bar u'+\bar u''+\frac12d^2f(\bar u)(\d_k\bar u,\d_k\bar u)\rangle_{L^2([0,1])}=\frac12\omega_0''(\ks)\ .
$$
The last equality comes differentiating twice the profile equation with respect to $k$ (see \cite{DSSS}).
\end{proof}

\begin{proof}[Proof of Theorem \ref{main}]

With estimates on  $\ks\partial_x \psi(t)$ from Proposition \ref{oldmain}, the estimate on $\tilde u$ follows
(recall Remark \ref{altconv})
from the bound
\be\label{weaker_refinedest}
\|\tilde u(\cdot- \psi(\cdot,t), t)-\bar u^{\ks}(\cdot)-\d_k\bar u^{\ks}(\cdot)\ks\d_x\psi(\cdot,t)\|_{L^p(\RM)}
\lesssim E_0 \ln(2+t)\ (1+t)^{-\frac{3}{4}}\
\ee
established in Proposition \ref{step} for $2\leq p\leq\infty$.

Using Duhamel's principle we may write
%$$
\be\label{temp}
k(t)=\sigma(t)\ks\partial_x h_0 + \int_0^t \sigma (t-s)
\partial_x (ek^2(s)))ds\ ,
\ee
%$$
where $\sigma$ is the constant-coefficient solution operator
defined in Proposition \ref{heatcomp}.
On the other hand, by Propositions
\ref{greenbds}, \ref{modprop}, and \ref{heatcomp},
\ba\label{repeat}
\ks\psi_x(t)
&=\ks\chi(t)\left(\partial_x s^{\rm p}(t) (d_0+h_0\bar u')
+
\int_0^t \partial_x s^{\rm p}(t-s) \N(s)  ds\right)
+(1-\chi(t))\ks\partial_x h_0,\\
&=\sigma(t)\ks\partial_x h_0 + \int_0^t \sigma (t-s) \partial_x (e(\ks\psi_x)^2(s))ds
+\tilde r(t),
\ea
where
$$
\|\tilde r(t)\|_{L^p(\RM)}\lesssim E_0\ln(2+t)\ (1+t)^{-\frac{1}{2}(1-1/p)-\frac{1}{2}},\quad 2\leq p\leq\infty\ .
$$
Here, we have used \eqref{heatspdiffx}
and
$$
\|\partial_x(\psi_x^2(s))\|_{L^1(\RM)\cap L^\infty(\RM)}\lesssim
\|\psi_x(s)\|_{L^2(\RM)\cap L^\infty(\RM)} \|\psi_{xx}(s)\|_{L^2(\RM)\cap L^\infty(\RM)}
\lesssim  (1+s)^{-1}
$$
to bound
$$
\begin{aligned}
\|\int_0^t \partial_x s^{\rm p}(t-s) f^{\rm p} \ks^2\psi_x(t)^2 ds
- &\int_0^t \sigma (t-s) \partial_x (e(\ks\psi_x)^2(s))ds\|_{L^p(\RM)}\\
\lesssim
&\int_0^t (1+t-s)^{\frac12(1-1/p)}(t-s)^{-\frac12}\|\partial_x
(\psi_x^2(s))\|_{L^1(\RM)\cap L^\infty(\RM)}ds\\
\lesssim
&\int_0^t (1+t-s)^{\frac12(1-1/p)}(t-s)^{-\frac12} (1+s)^{-1}ds.
\end{aligned}
$$

Thus, subtracting \eqref{temp} from \eqref{repeat}, and defining $\delta:=\ks\psi_x-k$, we have
\be\label{comp}
\delta(t)=
\int_0^t \sigma (t-s) \partial_x (e \delta(s) (k(s)+\ks
\psi_x (s)))ds +\tilde r(t).
\ee
Defining, for some $\eta>0$,
$$
\nu(t):=\sup_{p\in[2,\infty]}\sup_{0\le s\le t}\|\delta(s)\|_{L^p(\RM)}
(1+s)^{\frac{1}{2}(1-1/p)+\frac12 - \eta},
$$
we thus obtain by the standard heat bounds
$\|\sigma(t)\partial_x f\|_{L^p(\RM)}\le Ct^{-\frac12(1/q-1/p)-\frac12}\|f\|_{L^q(\RM)}$ when $1\le q\le p\le\infty$, that, for $2\le p\le\infty$,
$$
\begin{aligned}
\|\delta(t)\|_{L^p(\RM)}&\lesssim
\int_0^{t/2} (t-s)^{-\frac12(1-1/p)-\frac12}\|
e \delta(s) (k(s)+\ks \psi_x (s))\|_{L^1(\RM)} ds \\
&\quad +
\int_{t/2}^t (t-s)^{-\frac12(1/2-1/p)-\frac12}\|
e \delta(s) (k(s)+\ks \psi_x (s))\|_{L^2(\RM)} ds +\|\tilde r(t)\|_{L^p(\RM)}\\
&\lesssim
\int_0^{t/2} (t-s)^{-\frac12(1-1/p)-\frac12}
 \nu(t)E_0(1+s)^{-1+\eta}
ds\\
&\quad + \int_{t/2}^t (t-s)^{-\frac12(1/2-1/p)-\frac12}
 \nu(t)E_0(1+s)^{-\frac54+\eta}
ds\\
&\quad  + E_0\ln(2+t)\ (1+t)^{-\frac{1}{2}(1-1/p)-\frac{1}{2}}\\
&\lesssim
E_0(\nu(t)+1)(1+t)^{-\frac12(1-1/p)-\frac12+\eta}
,
\end{aligned}
$$
giving
$$
\nu(t)\le C_\eta E_0\left(1+\nu(t)\right),
$$
whence, if $E_0<1/(2C_\eta)$, we have (noting that $\nu$ by standard theory remains bounded) that $\nu(t)\le 2C_\eta E_0$. This gives
$\|k(t)-\ks\partial_x \psi(t)\|_{L^p(\RM)}=\|\delta(t)\|_{L^p(\RM)}
\le 2C_\eta (1+t)^{-\frac{1}{2}(1-1/p) -\frac12+\eta }$,
completing the result for $k-\ks\psi_x$. A similar computation yields the result for $h-\psi$.
\end{proof}

\medskip

{\bf Acknowledgement.} Thanks to Bj\"orn Sandstede for
pointing out the results of \cite{SSSU}.
Thanks also to the anonymous referees for several helpful suggestions
that greatly improved the exposition.

\appendix
\section{Asymptotic equivalence of scalar equations}\label{cl_proof}

Lemma \ref{cl_lemma} follows by a simplified (but also somewhat modified;
this does not immediately follow from the results stated in \cite{Ka,LZ})
version of the arguments used in \cite{Ka,LZ} to prove corresponding
but weaker versions in the system case.
We include a proof here, both for completeness, and to
motivate the more complicated comparison arguments appearing in the main body
of the paper.

\begin{proof}[Proof of Lemma \ref{cl_lemma}]
By the general, system, results of \cite{Ka},\footnote{
See \cite[Remark 4.2]{Ka} improving the result of \cite[Theorem 4.2]{Ka}
in the strictly parabolic case.
Bounds \eqref{kawbds} are proved by estimates similar to those
of this paper and \cite{JNRZ1}; see for example \eqref{simest} just below.}
provided $E_0:=\|k_0\|_{L^1(\RM)\cap H^3(\RM)}$ is sufficiently small, we have
for $\tilde k:=\kappa-\ks$ and $1\le p\le\infty$,
the ``heat-type'' bounds
\ba\label{kawbds}
\|k(t)\|_{L^p(\RM)},\,
\|\tilde k(t)\|_{L^p(\RM)} &\lesssim E_0 (1+t)^{-\frac{1}{2}(1-\frac{1}{p})},
\quad
\|k_x(t)\|_{H^1(\RM)},\,
\|\tilde k_x(t)\|_{H^1(\RM)} &\lesssim E_0 (1+t)^{-\frac34}.
\ea

Setting $\delta:=\tilde \kappa -\kappa$, we have, subtracting and rearranging,
$$
\delta_t - a \delta_x- d\delta_{xx}= \partial_x \CalF,\qquad
\CalF=O((|k|+|\tilde k|)\delta)
+O(\tilde k^3)+O(\tilde k \tilde k_x),
$$
with $ \delta|_{t=0}=0$, where
$a=-dq(\ks)$,
$d=d(\ks)$ are constant.
By Duhamel's formula,
$$
\delta (t)= \int_0^t \sigma(t-s)\d_x\CalF(s) ds,
$$
where $\sigma$ is the solution operator of the
convected heat equation $u_t -au_x-du_{xx}=0$.
Applying the standard heat bounds
$\|\sigma(t)\partial_x^r f\|_{L^p(\RM)}\lesssim
t^{-\frac{1}{2}(\frac1q-\frac{1}{p})-\frac{r}{2}}\|f\|_{L^q(\RM)}$,
$1\le q\le p\le \infty$,
together with
$$
\|\CalF(t)\|_{L^q(\RM)}\ \lesssim\
E_0 (1+t)^{-\frac12(1-1/q)-\frac14}(\|\delta(t)\|_{L^2(\RM)}
+\|\tilde k_x(t)\|_{L^2(\RM)}) +E_0^2
(1+t)^{-\frac12(1-1/q)-1},
$$
$1\le q\le 2,$
we find, defining $\nu(t):=
\sup_{0\le s\le t}\|\delta(s)\|_{L^2(\RM)}(1+s)^{\frac34-\eta}$,
that, for all $1\le p\le \infty$,
\be\label{simest}
\begin{aligned}
\|\delta(t)\|_{L^p(\RM)}&\lesssim
\int_0^{t/2} (t-s)^{-\frac12(1-1/p)-\frac12}\| \CalF (s)\|_{L^1(\RM)} ds\\
&\quad +
\int_{t/2}^t (t-s)^{-\frac12(1/(\min(2,p))-1/p)-\frac12}
\| \CalF (s)\|_{L^{\min(2,p)}(\RM)} ds \\
&\lesssim
\int_0^{t/2} (t-s)^{-\frac12(1-1/p)-\frac12} (\nu(t)E_0+E_0^2)(1+s)^{-1+\eta}ds\\
&+\int_{t/2}^t (t-s)^{-\frac12(1/(\min(2,p))-1/p)-\frac12}(\nu(t)E_0+E_0^2)(1+s)^{-1+\eta-\frac12(1-1/(\min(2,p)))}ds\\
&\lesssim
E_0(E_0 +\nu(t)) (1+t)^{-\frac12(1-1/p)-\frac12 +\eta},
\end{aligned}
\ee
whence $ \nu(t)\le C_\eta E_0\left(E_0+\nu(t)\right) $.
This implies that
$\nu(t)\le 2C_\eta E_0^2$ for $E_0<1/(2C_\eta)$, giving
$$
\|\delta(t)\|_{L^p(\RM)}\leq 2C_\eta E_0^2
(1+t)^{-\frac12(1-1/p)-\frac12 +\eta} ,
\qquad 1\le p\le \infty.
$$

Finally, let
$\phi(x,t)=\frac{1}{\sqrt{t}}\bar\phi\left(\frac{x+at}{\sqrt{t}}\right)$
define a self-similar solution of Burgers equation
\eqref{mainwhitham} such that $\int \bar \phi=\int k_0$. Defining
$\tilde \delta(t):= k(t)-\phi(t+1)$,
we have by the $L^1$ and $L^1$ first moment assumptions that, for $1\leq p\leq\infty$,
$$
\|\sigma(t)\tilde\delta(0)\|_{L^p(\RM)}\lesssim E_1 (1+t)^{-\frac{1}{2}(1-1/p)-\frac{1}{2}}.
$$
Expressing by Duhamel's formula
$
\tilde \delta (t)=\sigma(t)\tilde \delta(0)+ \int_0^t \sigma(t-s)\d_x\tilde\CalF(s) ds,
$
where
$\tilde\CalF(t)=O((|k(t)|+|\phi(t+1)|)\tilde\delta(t))$,
and, for any $\eta'>0$,
estimating as before, we obtain by a contraction argument
similar to the above that
$
\|\tilde\delta(t)\|_{L^p(\RM)}\ \lesssim\
E_1\ (1+t)^{-\frac12(1-1/p)-\frac12+\eta'},\qquad1\le p\le \infty,
$
provided $E_1$ is small enough (depending on $\eta'$). Indeed one may prove for
$$
\tilde\nu(t):=\sup_{p\in[1,\infty]}\sup_{0\le s\le t}\|\tilde\delta(s)\|_{L^p(\RM)}(1+s)^{\frac12(1-1/p)+\frac12-\eta'}
$$
that $\tilde\nu(t)\leq C'_{\eta'}E_1(1+\tilde\nu(t))$ and conclude
that  $\tilde\nu(t)\le 2C'_{\eta'}E_1$ if $C'_{\eta'}E_1<1/2$.
\end{proof}

\br\label{lastrmk}
\textup{
The order $O(\d_x(k\delta)) \sim O(\d_x(kk_x))\sim O(\d_x(k^3))$ of neglected terms
in the argument above is consistent with the order neglected throughout the
paper. Indeed, as emphasized by Remark \ref{altconv}, the true local wave number is $\ks\tilde \Psi_x$ (where $\tilde \Psi(\cdot,t)$ is the inverse of $y\mapsto X(y,t):=y-\psi(y,t)$) which differs from $\ks(1+\psi_x)$ by an
$O(\|\psi_x\|^2)$ term since
$$
\tilde\Psi_x(x,t)-[1+\psi_x(x,t)]=\psi_x(\tilde\Psi(x,t),t)-\psi_x(\tilde\Psi(x,t)-\psi(\tilde\Psi(x,t),t),t)\tilde\Psi_x(x,t)(1-\psi_x(\tilde\Psi(x,t),t)).
$$
With a slight bit of additional effort, one can check that $\ks\tilde \Psi_x$ satisfies the Whitham equation or its quadratic approximant only up to a truncation of the order $\partial_x^2 (\psi_x^2)\sim \partial_x (k k_x)$. Once integrated, this means that the formal Whitham equation \eqref{whitham} derived for
$\ks(1+\psi_x)$ can be expected to hold for $\ks\tilde\Psi_x$ only to the
(asymptotically equivalent) quadratic order considered here.
In particular, if we were to derive higher order expansions for the equation
satisfied by $\tilde\Psi_x$, as we could in principle do, {\it these
would not in general agree with the Whitham equation \eqref{whitham}}. In any case, while equation \eqref{mainwhitham}, most easily obtained by switching to a comoving reference frame $x\to x-at$, is designed to match, about a given wave, exactly the order of description we attain here, equation \eqref{whitham} is just a convenient way to piece together the two first orders of a nonlinear
WKB expansion; thus \eqref{whitham} could in principle be meaningful even when dealing with a modulated background wave (instead of a true wave) but should not improve in accuracy its quadratic approximant about a given wave. For further discussion of this and related issues, see \cite{DSSS,NR1,NR2}.
}
\er

\end{document}